\documentclass[11pt]{article}
\usepackage{amsmath,amsfonts,amsthm,amssymb,color, fullpage}

\usepackage{url}  
\newtheorem{theorem}{Theorem}[section]
\newtheorem{corollary}[theorem]{Corollary}
\newtheorem{lemma}[theorem]{Lemma}

\newtheorem{claim}[theorem]{Claim}

\theoremstyle{definition}
\newtheorem{definition}[theorem]{Definition}

\def\defeq{\stackrel{\mathrm{def}}{=}}
\def\setof#1{\left\{#1  \right\}}
\def\sizeof#1{\left|#1  \right|}

\def\union{\cup}

\newcommand{\R}{\mathbb{R}}
\newcommand{\AND}{\quad \text{and} \quad}
\newcommand{\mydet}[1]{\det\left[#1\right]}
\newcommand{\tr}{\mathrm{Tr}}
\newcommand{\X}{\mathbf{r}}

\DeclareMathOperator*{\E}{\mathbb{E}}

\newcommand{\smin}{\sigma_\mathrm{min}}
\newcommand{\sr}{\mathrm{srank}}

\newcommand{\Tr}[1]{\mathrm{Tr}\left[ #1 \right]}
\newcommand{\deriv}{\partial_{x}}
\newcommand{\kderiv}[1]{\partial_x^{#1}}

\newcommand{\lmin}{\lambda_{min}}

\newcommand{\amin}[1]{\alpha \mathsf{-}\mathsf{min}\left(#1 \right)}
\newcommand{\aminName}{\alpha \mathsf{-}\mathsf{min}}

\newcommand{\rootnode}{\emptyset}

\def\norm#1{\left\| #1 \right\|}

\begin{document}

\title{
Interlacing Families III: Sharper Restricted Invertibility Estimates
\thanks{
Many of the results in this paper were first announced in lectures by the authors in 2013.
This research was partially supported by NSF grants CCF-0915487, CCF-1111257, CCF-1562041, CCF-1553751, DMS-0902962, DMS-1128155 and DMS-1552520, a Sloan Research Fellowship to Nikhil Srivastava
  a Simons Investigator Award to Daniel Spielman, and a MacArthur Fellowship.
}}

\author{
Adam W. Marcus\\
Princeton University\\
\and
Daniel A. Spielman \\ 
Yale University
\and 
Nikhil Srivastava\\
UC Berkeley
}

\maketitle

\begin{abstract}
We use the method of interlacing families of polynomials to derive a simple
  proof of Bourgain and Tzafriri's Restricted Invertibility Principle,
  and then to sharpen the result in two ways.
We show that the stable rank can be replaced by the Schatten 4-norm stable rank
  and that tighter bounds hold when the number of columns in the matrix under consideration
  does not greatly exceed its number of rows.
Our bounds are derived from an analysis of smallest zeros of Jacobi and associated Laguerre polynomials.
\end{abstract}

\section{Introduction}
The Restricted Invertibility
  Principle of Bourgain and Tzafriri \cite{BourgainTzafriri}
  is a quantitative generalization of the assertion that
  the rank of a $d\times m$ matrix $B$ is the maximum number of linearly
  independent columns that it contains.  
It says that if a matrix $B$ has high {\em stable rank}:
\[
\sr(B) \defeq \frac{ \norm{B}_{F}^{2}}{\norm{B}_{2}^{2}},
\]
then it must contain a large column
submatrix $B_{S}$, of size $d \times |S|$, with large
{\em least singular value}, defined as:
\[
\smin(B_{S}) \defeq \min_{x\neq 0}\frac{\norm{B_{S} x}}{\norm{x}}.
\]
The least singular value of $B$ is a measure of how far the matrix is from being singular.
Bourgain and Tzafriri's result was strengthened in the works of \cite{vershynin2001john, SpielmanSrivastava12,
youssef, naoryoussef}, and has since been a useful tool in Banach space theory, data
mining, and more recently theoretical computer science.

Prior to this work, the sharpest result of this type was the following theorem
of Spielman and Srivastava \cite{SpielmanSrivastava12}:
\begin{theorem}\label{thm:ss}
Suppose $B$ is an $d \times m$ matrix and $k\le \sr(B)$ is an integer.
Then there exists a subset $S\subset [m]$ of size $k$ such that
\begin{equation}\label{eqn:ssthm}\smin(B_S)^2\ge
\left(1-\sqrt{\frac{k}{\sr(B)}}\right)^2\frac{\|B\|_F^2}{m}.\end{equation}
\end{theorem}
Note that when $k$ is proportional to $\sr(B)$, Theorem \ref{thm:ss} produces a submatrix whose squared
least singular value is at least a constant times the average squared norm of
the columns of $B$, a bound which cannot be improved even for $k=1$. 
Thus, the theorem tells us that the columns of $B_S$ are ``almost
orthogonal'' in that they have least singular value comparable to the average
squared norm of the vectors individually.

To understand the form of the bound in \eqref{eqn:ssthm}, consider the case when
$B B^{T}=I_{d}$, which is sometimes called the ``isotropic'' case. 
In this situation
  we have $\sr(B)=d$, and the right hand side of \eqref{eqn:ssthm} becomes
\begin{equation}\label{eqn:ssiso}
\left(1-\sqrt{\frac{k}{d}}\right)^{2}\frac{d}{m}. 
\end{equation}
The number $(1-\sqrt{{k}/{d}})^{2}$ may seem familiar, and arises in the following
	two contexts.
\begin{enumerate} 
\item It is an asymptotically sharp lowerbound on the least zero of the
	associated Laguerre polynomial $L^{d-k}_{k}(x)$, after an appropriate
	scaling.
In Section~\ref{sec:laguerre} we derive the isotropic case of
  Theorem~\ref{thm:ss} from this fact, using the method of interlacing families of polynomials
\footnote{A version of this proof appeared in the authors' survey paper 
\cite{ICM14}}.

\item It is the lower edge of the support of the Marchenko-Pastur distribution
	\cite{marchenkopastur}, which is the limiting spectral distribution of a
	sequence of random matrices $G_d G_d^T$ where $G_d$ is $k\times d$ with
	appropriately normalized i.i.d. Gaussian entries, as $d \rightarrow
		\infty$ with $k/d$ fixed \cite{marchenkopastur}. This convergence result along with large
	deviation estimates may be used to show that it is not possible to obtain a bound of 
	\begin{equation}\label{eqn:lowerb}
  \left(1-\sqrt{\frac{k}{\sr(B)}}+\delta\right)^2\frac{\|B\|_F^2}{m}\end{equation}
		in Theorem \ref{thm:ss} for any constant $\delta>0$, when $m$ goes to infinity significantly faster then $d$. 
Thus, the bound of Theorem \ref{thm:ss} is asymptotically sharp.  See \cite{SrivastavaBlogRI} for details.
\end{enumerate}

In Section~\ref{sec:laguerre} we present a proof of Theorem
  \ref{thm:ss} in the isotropic case using the method of interlacing polynomials.   
This proof considers the expected characteristic polynomial of
  $B_{S} B_{S}^{T}$ for a randomly chosen $S$.
In this first proof, we choose $S$ by sampling $k$ columns
  with replacement.
This seems like a suboptimal thing to do since it may select
  a column twice (corresponding to a trivial bound of $\sigma_{min}(B_S)=0$), but
it allows us to easily prove that the expected
  characteristic polynomial is an associated Laguerre polynomial
  and that the family of polynomials that arise in the expectation form
  an \textit{interlacing family}, which we define below.
Because these polynomials form an interlacing family, there is some polynomial
  in the family whose $k$th largest zero is at least the $k$th
  largest zero of the expected polynomial.
The bound \eqref{eqn:ssiso} then follows from lower bounds on the zeros
  of associated Laguerre polynomials.

In Section~\ref{sec:schatten},
  we extend this proof technique to show 
  that Theorem \ref{thm:ss} remains true in the nonisotropic case.
In addition, we show a bound that replaces the stable rank   
  with a Schatten 4-norm stable rank, defined by
\[
 \sr_{4}(B) \defeq \frac{\norm{B}_{2}^4}{\norm{B}_{4}^4},
\]
where $\norm{B}_{p}$ denotes the Schatten $p$-norm, i.e., the $\ell_p$ norm of the
singular values of $B$. 
That is, 
\[
 \sr_4(B)= \frac{(\sum_{i}\sigma_i^2)^2}{\sum_{i}\sigma_i^4},
\]
where $\sigma_{1} \le \ldots \le \sigma_{d}$ are the singular values of $B$.
As
\[
 \sr_4(B)= \frac{(\sum_{i}\sigma_i^2)^2}{\sum_{i}\sigma_i^4}\ge
	 \frac{(\sum_{i}\sigma_i^2)^2}{\sigma_1^2\sum_{i\le
	 n}\sigma_i^2}=\sr(B),
\]
this is a strict improvement on Theorem \ref{thm:ss}.
The above inequality is far from tight when $B$ has many moderately large
  singular values. 
In Section~\ref{sec:alg} we give a polynomial time algorithm for 
  finding the subset guaranteed to exist by Theorem 
  \ref{thm:ss}.

In Section~\ref{sec:jacobi},
  we improve on Theorem \ref{thm:ss} in the isotropic case
  by sampling the sets $S$ without replacement.
We show that the resulting expected characteristic polynomials
  are scaled Jacobi polynomials.
We then derive a new bound on the smallest zeros of Jacobi polynomials
  which implies that there exists a set $S$ of $k$ columns for which
\begin{equation}\label{eqn:jacobiBound}
\smin(B_S)^2 \geq  
\frac{\left(\sqrt{d (m-k)} - \sqrt{k(m-d)}\right)^2}{m^2}.
\end{equation}
As
\[
\frac{\left(\sqrt{d(m-k)} - \sqrt{k(m-d)}\right)^2}{m^2}
\geq
\left(1+\frac{\sqrt{d k}}{m}\right) 
\left(1-\sqrt{\frac{k}{d}}\right)^2\frac{d}{m}
\]
this improves on Theorem \ref{thm:ss} by a constant factor
  when $m$ is a constant multiple of $d$.
Note that this does not contradict the lower bound 
  \eqref{eqn:lowerb} from \cite{SrivastavaBlogRI}, which requires that $m \gg 
  d$. 

A number of the results in this paper require a bound on the smallest root of a 
polynomial.
In order to be as self contained as possible, we will either prove such bounds 
directly or take the best known bound directly from the literature.
It is worth noting, however, that a more generic way of proving each of these bounds
  is provided by the framework of polynomial convolutions developed in 
\cite{MSSfinite}.
The necessary inequalities in \cite{MSSfinite} are known to be asymptotically 
  tight (as shown in \cite{Adamfinitefree}) and in some cases improve on the 
  bounds given here.

\section{Preliminaries}
\subsection{Notation}
We denote the Euclidean norm of a vector $x$ by $\norm{x}$.
We denote the operator norm by:
\[
	\norm{B}_{2} =\norm{B}_\infty \defeq \max_{\norm{x} = 1}\norm{B x}.
\]
This also equals the largest singular value of the matrix $B$.
The Frobenius norm of $B$, also known as the Hilbert-Schmidt norm
  and written $\norm{B}_{2}$,
  is the square root of the sum of the squares of the singular values of $B$.
It is also equal to the square root of the sum of the squares of the entries of $B$.

For a real rooted polynomial $p$, we let $\lambda_{k}(p)$ denote the $k$th largest zero of $p$.
When we want to refer to the smallest zero of a polynomial $p$ without specifying its degree,
  we will call it $\lmin (p)$. 
We define the $\ell$th elementary symmetric function of a matrix $A$ with eigenvalues
  $\lambda_{1}, \dotsc , \lambda_{d}$ to be the $\ell$th elementary symmetric function of those eigenvalues:
\[
e_\ell(A)=\sum_{S\subset [d],|S|=\ell} \prod_{i\in S}\lambda_i.
\]
Thus, the characteristic polynomial of $A$ may be expressed as
\[
\mydet{x I - A}
  = \sum_{ \ell =0}^{d} (-1)^{\ell}e_{\ell}(A)x^{d-\ell}.
\]
By inspecting the Leibniz expression for the determinant in terms of permutations, it is
  easy to see that the $e_{\ell}$ may also be expanded in terms of minors.
The Cauchy-Binet identity  says that for every  $d \times m$ matrix $B$
\[
e_{\ell}(B^TB) = \sum_{T \in \binom{[m]}{\ell}} \mydet{B_{T}^TB_{T}},
\]
where $T$ ranges over all subsets of
  size $\ell$ of indices in $[m] \defeq \setof{1,\dotsc , m}$, and
$B_{T}$ denotes the $d \times \ell$ matrix formed by the columns of $B$
  specified by $T$.

We will use the following two formulas to calculate determinants
  and characteristic polynomials of matrices.
You may prove them yourself, or find proofs in
  \cite[Chapter 6]{Meyer} or \cite[Section 15.8]{Harville}.

\begin{lemma}\label{lem:matrix_det}
For any invertible matrix $A$ and vector $u$
\[
\mydet{A + uu^T} 
= \mydet{A}(1 + u^TA^{-1}u)
= \mydet{A}(1 + \Tr{A^{-1} u u^{T}}).
\]
\end{lemma}

\begin{lemma}[Jacobi's formula]
\label{lem:matrix_deriv}
For any square matrices $A, B$,
\[
\deriv \mydet{xA + B} 
= \mydet{xA + B}\Tr{A(xA + B)^{-1}}
\]
\end{lemma}

We also use the following consequence of these formulas that was derived in \cite[Lemma 4.2]{IF2}.
\begin{lemma}\label{lem:randomUpdate}
For every square matrix $A$ and random vector $\X$,
\[
\E \mydet{A - \X \X^{T}} = (1 - \partial_{t} ) \mydet{A + t \E \X \X^{T}} \big|_{t = 0}.
\]
\end{lemma}

\subsection{$\aminName $}
We bound the zeros of the polynomials we construct 
  by using the barrier function arguments developed in
  \cite{BatsonSpielmanSrivastava}, \cite{IF2}, and
  \cite{MSSfinite}.
For $\alpha > 0$, we define the {\em $\aminName$} of a polynomial $p(x)$ to
  be the least root of $p(x) + \alpha \partial_{x} p(x)$,
  where $\partial_{x}$ indicates partial differentiation with respect
  to $x$.
We sometimes write this in the compact form
\[
	\amin{p(x)} \defeq  \lmin (p(x) + \alpha p'(x)).
\]
We may also define this in terms of the	\textit{lower barrier function}
\[
  \Phi_{p}(x) \defeq  -\frac{p'(x)}{p(x)} 
\]
by observing
\[
  \amin{p} = \min  \setof{z : \Phi_{p}(z) = 1/\alpha}.
\]
As 
\[
  \Phi_{p}(x) =  \sum_{i=1}^{d} \frac{1}{\lambda_{i} - x},
\]
we see that $\amin{p}$ is less than the least root of $p(x)$.

The following claim is elementary.

\begin{claim}\label{clm:aminMonomial}
For $\alpha > 0$,
\[
  \amin{x^{k}} = - k \alpha .
\]
\end{claim}

\subsection{Interlacing Families}\label{ssec:interlacing}

We use the method of interlacing families of polynomials developed in
  \cite{IF1,IF2}
  to relate the zeros of sums of polynomials to individual polynomials
  in the sum.
The results in Section~\ref{sec:jacobi} will require 
 the following variant of the definition,
  which is more general than the ones we have used previously.
\begin{definition}\label{def:if} An {\em interlacing family} consists of a finite rooted tree
	$T$ and a labeling of the nodes $v \in T$ by monic real-rooted
	polynomials $f_{v} (x)\in\R[x]$, with two properties:
\begin{enumerate}
	\item [a.] Every polynomial $f_{v} (x)$ corresponding to a non-leaf node $v$ is
		a convex combination of the polynomials corresponding to the
		children of $v$.
	\item [b.] For all nodes $v_1,v_2\in T$ with a common parent, all convex
		combinations of $f_{v_1}(x)$ and $f_{v_2}(x)$ are
		real-rooted.\footnote{This condition implies that all convex
			combinations of all the children of a node are
			real-rooted; the equivalence is discussed in
		\cite{ICM14}.}
\end{enumerate}
We say that a set of polynomials is an \textit{interlacing family} if they are the labels of the leaves of such a tree.
\end{definition}
In the applications in this paper, the leaves of the tree will naturally
correspond to elements of a probability space, and the internal nodes will
correspond to conditional expectations of the corresponding polynomials over this probability space.

In Sections \ref{sec:laguerre} and \ref{sec:schatten}, as in \cite{IF1,IF2}, we consider
  interlacing families in which the nodes of the tree at distance $t$ from the root 
  are indexed by sequences
  $s_{1}, \dotsc , s_{t} \in [m]^{t}$.
We denote the empty sequence and the root node of the tree by $\rootnode$.

The leaves of the tree correspond to sequences of length $k$, and each is labeled
  by a polynomial $f_{s_{1}, \dotsc , s_{k}}(x)$.
Each intermediate node is labeled by the average of the polynomials labeling its children. 
So, for $t < k$
\[
  f_{s_{1}, \dotsc , s_{t}} (x)
=
 \frac{1}{m}  \sum_{j=1}^{m}   f_{s_{1}, \dotsc , s_{t}, j} (x)
=
 \frac{1}{m^{k-t}}  \sum_{s_{t+1}, \dotsc , s_{k}} f_{s_{1}, \dotsc , s_{k}} (x), 
\]
and
\[
 \frac{1}{m^{k}}  f_{\rootnode}(x) = \sum_{s_{1}, \dotsc , s_{k} \in [m]^{k}} f_{s_{1}, \dotsc , s_{k}} (x).
\]

A fortunate choice of polynomials labeling the leaves yields an interlacing family. 
\begin{theorem}[Theorem 4.5 of \cite{IF2}]\label{thm:cartesianInterlacing}
Let $u_{1}, \dotsc , u_{m}$ be vectors in $\R^{d}$
  and let
\[
  f_{s_{1}, \dotsc , s_{k}} \defeq  \mydet{x I - \sum_{i = 1}^{k} u_{s_{i}} u_{s_{i}}^{T}}.
\]
Then, these polynomials form an interlacing family.
\end{theorem}

Interlacing families are useful because they allow us to relate the zeros
  of the polynomial labeling the root to those labeling the leaves.
In particular, we will prove the following slight generalization of
  Theorem 4.4 of \cite{IF1}.

\begin{theorem}\label{thm:kthRoot}
Let $f$ be an interlacing family of degree $d$ polynomials with root
  labeled by $f_{\rootnode}(x)$ and leaves by $\{f_{\ell}(x)\}_{\ell\in L}$.
Then for all indices $1 \leq j \leq n$, there exist leaves $a$ and
$b$ such that
\begin{equation}\label{eqn:kthRoot}
	\lambda_j(f_{a}) \geq \lambda_j(f_\rootnode) \geq
	\lambda_j(f_{b}).
\end{equation}
\end{theorem}

To prove this theorem, we first explain why we call these families ``interlacing''.

\begin{definition}\label{def:interlacing}
We say that a polynomial $g(x) = \prod_{i=1}^{d+1} (x - \alpha_{i})$ \emph{interlaces} a polynomial 
  $f(x) = \prod_{i=1}^{d} (x - \beta_{i})$ if
\[
  \alpha_{1} \geq \beta_{1} \geq \alpha_{2} \geq \beta_{2} \geq \alpha_{3} \geq \dotsb \geq
  \alpha_{d}\geq \beta_{d} \geq \alpha_{d+1}. 
\]
We say that polynomials $f_{1}, \dotsc , f_{m}$ have a \emph{common interlacing}
  if there is a single polynomial $g$ that interlaces $f_{i}$ for each $i$.
\end{definition}

The common interlacing assertions in this paper stem from the following fundamental example.
\begin{claim}\label{clm:cauchyInterlacing}
Let $M$ be a $d$ dimensional symmetric matrix and let $u_{1}, \dotsc , u_{k}$ be vectors in $\R^{d}$.
Then the polynomials 
\[
f_{j}(x) = \mydet{xI - M - u_{j}u_{j}^{T}}
\]
have a common interlacing.
\end{claim}
\begin{proof}
Let $g_{0}(x) = \mydet{xI - M} = \prod_{i=1}^{d}(x - \alpha_{i})$.
For any $j$, let $f_{j}(x) = \prod_{i=1}^{d}(x - \beta_{i})$.
Cauchy's interlacing theorem tells us that
\[
  \beta_{1} \geq \alpha_{1} \geq \beta_{2} \geq \alpha_{2} \geq \dotsb \geq
  \beta_{d}\geq \alpha_{d}.
\]
So, for sufficiently large $\alpha$, $(x - \alpha) g_{0}(x)$  interlaces each $f_{j}(x)$.
\end{proof}

The connection between interlacing and the real-rootedness of convex combinations is given
  by the following theorem (see \cite{Dedieu}, \cite[Theorem~$2'$]{Fell}, and \cite[Theorem 3.6]{ChudnovskySeymour}.

\begin{theorem}\label{thm:Fisk}
Let $f_{1},\ldots, f_{m}$ be real-rooted (univariate) polynomials of the same degree with
positive leading coefficients. Then $f_{1},\ldots, f_{m}$ have a common interlacing if
and only if $\sum_{i=1}^{m}\mu_{i} f_{i}$ is real rooted for all convex
combinations $\mu_{i} \ge 0, \sum_{i=1}^{m} \mu_{i}=1$.
\end{theorem}

Theorem \ref{thm:kthRoot} follows from an inductive application of the following lemma, which
  generalizes the case $k=1$ which was proven in \cite{IF1,IF2}.
\begin{lemma}\label{lem:interlacingRoots}
Let $f_{1}, \dotsc , f_{m}$ be real-rooted degree $d$ polynomials that have a common interlacing.
Then for every index $1 \leq j \leq d$ and
  for every nonnegative $\mu_{1}, \dotsc , \mu_{m}$ such that
  $\sum_{i=1}^{m} \mu_{i}=1$,
there exist an $a$ and a $b$ so that
\[
  \lambda_{j}(f_{a})
\geq 
  \lambda_{j} \left(\sum_{i} \mu_{i} f_{i} \right)
\geq 
  \lambda_{j} (f_{b}).
\]
\end{lemma}
\begin{proof}
By restricting our attention to the polynomials $f_{i}$ for which $\mu_{i}$ is positive,
  we may assume without loss of generality that each $\mu_{i}$ is positive for every $i$.
Define
\[
  f_{\rootnode}(x) = \sum_{i} \mu_{i} f_{i} (x)
\]
and let
\[
  f_{\rootnode}(x) = \prod_{i=1}^{d} (x - \beta_{i}).
\]
We seek $a$ and $b$ for which $\lambda_{j}(f_{a}) \geq \beta_{j} \geq \lambda_{j}(f_{b})$.

Let 
\[
g(x) = \prod_{i=1}^{d+1} (x - \alpha_{i})
\]
be  a polynomial that interlaces every $f_{i}$.
As each $f_{i}$ has a positive leading coefficient, we know that
  $f_{i}(\alpha_{k})$ is at least $0$ for $k$ odd and at most $0$ 
  for $k$ even.

We first consider the case in which $f_{1}, \dotsc , f_{m}$ do not have any zero in common.
In this case, $\alpha_{1} > \alpha_{2} > \dotsb  > \alpha_{d+1}$: if some $\alpha_{k} = \alpha_{k+1}$
  then $f_{i}(\alpha_{k}) = 0$ for all $i$.
Moreover, there must be some $i$ for which $f_{i}(\alpha_{k})$ is nonzero.
As all the $\mu_{i}$ are positive,
  $f_{\rootnode}(\alpha_{k})$ is positive for $k$ odd and negative for $k$ even.

As there must be some $i$ for which $f_{i}(\beta_{j}) \not = 0$, there must be an
  $a$  for which $f_{a}(\beta_{j}) < 0$ and a $b$ for which $f_{b}(\beta_{j}) > 0$.
We now show that if $j$ is odd, then
  $\lambda_{j}(f_{a}) \geq \beta_{j} \geq \lambda_{j}(f_{b})$.
As  $f_{a}(\alpha_{j})$ is nonnegative,
  $f_{a}$ must have a zero between $\beta_{j}$ and $\alpha_{j}$.
As $f_{a}$ interlaces $g$, this is the $j$th largest zero of $f_{a}$.
Similarly, the nonpositivity of $f_{b}(\alpha_{j+1})$ implies that
  $f_{b}$ has a zero between $\alpha_{j+1}$ and $\beta_{j}$.
This must be the $j$th largest zero of $f_{b}$.

The case of even $j$ is symmetric, except that we
  reverse the choice of $a$ and $b$.

We finish by observing that it suffices to consider the case in which
   $f_{1}, \dotsc , f_{m}$ do not have any zero in common.
If they do, we let $f_{0}(x)$ be their greatest common divisor,
  define $\hat{f}_{i}(x) = f_{i}(x) / f_{0}(x)$, and observe
  that $\hat{f}_{1}, \dotsc , \hat{f}_{m}$ do not have any common zeros.
Thus, we may apply the above argument to these polynomials.
As multiplying all the polynomials by $f_{0}(x)$ adds the same zeros to
  for $f_{1}, \dotsc , f_{m}$ and $f_{\rootnode}$, the theorem holds
  for these as well.
\end{proof}

\begin{proof}[Proof of Theorem \ref{thm:kthRoot}]
For every node $v$ in the tree defining an interlacing family,
  the subtree rooted at $v$ and the polynomials on the nodes of that tree
  form an interlacing family of their own.
Thus, we may prove the theorem by induction on the height of the tree.
Lemma \ref{lem:interlacingRoots} handles trees of height $1$.

For trees of greater height, Lemma \ref{lem:interlacingRoots} tells us that
  there are children of the root $v_{\hat{a}}$ and $v_{\hat{b}}$ that satisfy
  \eqref{eqn:kthRoot}.
If $v_{\hat{a}}$ is not a leaf, then it is the root of its own interlacing family
  and Lemma \ref{lem:interlacingRoots} tells this family has a leaf node $v_{a}$
  for which
\[
  \lambda_{j}(v_{a})  \geq    \lambda_{j}(v_{\hat{a}})  \geq \lambda_{j}(f_{\rootnode}).
\]
The same holds for $v_{\hat{b}}$.
\end{proof}

\section{The Isotropic Case with Replacement: Laguerre Polynomials}
\label{sec:laguerre}

We now prove Theorem \ref{thm:ss} in the isotropic case.
Let the columns of $B$ be the vectors $u_{1}, \dotsc , u_{m} \in \R^{d}$.
The condition $B B^{T} = I_{d}$ is equivalent to $\sum_{i} u_{i} u_{i}^{T} = I_{d}$.
For a set $S$ of size $k < d$,
\[
 \smin (B_{S})^{2} = \lambda_{k}(B_{S}^{T} B_{S})
 = \lambda_{k} (B_{S} B_{S}^{T})
 = \lambda_{k} \left(\sum_{i\in S} u_{i} u_{i}^{T} \right).
\]

We now consider the expected characteristic polynomial of the sum of the outer
  products of $k$ of these vectors chosen uniformly at random, with replacement.
We indicate one such polynomial by a vectors of indices such as
  $(s_{1}, \dotsc , s_{k}) \in [m]^{k}$,
  where we recall $[m] = \setof{1, \dotsc, m}$.
These are the leaves of the tree in the interlacing family:
\[
  f_{s_{1}, \dotsc , s_{k}}(x) \defeq  \mydet{x I - \sum_{i = 1}^{k} u_{s_{i}} u_{s_{i}}^{T}}.
\]
As in Section \ref{ssec:interlacing}, the intermediate nodes in the tree are labeled by
  subsequences of this form, and the polynomial at the root of the tree is
\[
  f_{\rootnode}(x) = \frac{1}{m^{k}} \sum_{s_{1}, \dotsc , s_{k} \in [m]^{k}}   f_{s_{1}, \dotsc , s_{k}}(x).
\]

We now derive a formula for $f_{\rootnode}(x)$.

\begin{lemma}
\label{lem:laguerreFormula}
\[
f_{\rootnode}(x)
= \left(1 - \frac{1}{m} \deriv \right)^{k} x^{d}.
\]
\end{lemma}
\begin{proof}
For every $0 \leq t \leq k$, define
\[
  g_{t} = \frac{1}{m^{t}} \sum_{s_{1}, \dotsc , s_{t} \in [m]^{t}}   
  \mydet{ x I - \sum_{i = 1}^{k} u_{s_{i}} u_{s_{i}}^{T}} 
\]
We will prove by induction on $t$ that
\[
g_{t} = \left(1 - \frac{1}{m} \deriv \right)^{t} x^{d}.
\]
The base case of $t = 0$ is trivial.
To establish the induction, we use Lemma \ref{lem:matrix_det},
   the identity $\sum_{j} u_{j} u_{j}^{T} = I$, and Lemma \ref{lem:matrix_deriv}  
   to compute
\begin{align*}
g_{t+1}(x) 
&= \frac{1}{m^{t+1}} \sum_{s_1, \dots, s_t} \sum_j \mydet{xI - \sum_{i = 1}^t u_{s_i} u_{s_i}^T - u_ju_j^T}
\\&= \frac{1}{m^{t+1}} \sum_{s_1, \dots, s_t} \sum_j \mydet{xI - \sum_{i = 1}^t u_{s_i} u_{s_i}^T}\left(1 -  \Tr{\left(xI - \sum_{i = 1}^t u_{s_i} u_{s_i}^T\right)^{-1} u_ju_j^T} \right)
\\&= \frac{1}{m^{t+1}} \sum_{s_1, \dots, s_t} \mydet{xI - \sum_{i = 1}^t u_{s_i} u_{s_i}^T}\left(m -  \Tr{\left(xI - \sum_{i = 1}^t u_{s_i} u_{s_i}^T\right)^{-1} I} \right)
\\&= g_{t} (x) - \frac{1}{m^{t+1}} \sum_{s_1, \dots, s_t} \mydet{xI - \sum_{i = 1}^t u_{s_i} u_{s_i}^T}\Tr{\left(xI - \sum_{i = 1}^t u_{s_i} u_{s_i}^T\right)^{-1} }
\\&= g_{t} - \frac{1}{m^{t+1}} \sum_{s_1, \dots, s_t} \deriv  \mydet{xI - \sum_{i = 1}^t u_{s_i} u_{s_i}^T} 
\\&= g_{t}(x) - \frac{1}{m} \deriv g_{t}(x)
\\&= (1 - \frac{1}{m} \deriv ) g_{t}(x).
\end{align*}
\end{proof}

For $d \geq k$ the polynomial $f_{\rootnode}(x)$ 
  is divisible by $x^{d-k}$.
So, the $k$th largest root of $f_{\rootnode}(x)$
  is equal to the smallest root of
\[
  x^{-(d-k)} f_{\rootnode}(x) =
  x^{-(d-k)} \left(1 - \frac{1}{m} \deriv \right)^{k} x^{d}.
\]

To bound the smallest root of this polynomial, we observe that it is a slight transformation
  of an associated Laguerre polynomial.
We use the definition of the associated Laguerre polynomial of degree $n$ and parameter $\alpha$, 
  $L_{n}^{(\alpha)}$, given
  by Rodrigues' formula \cite[(5.1.5)]{Szego}
\[
  L_{n}^{(\alpha )}(x) = 
  e^{x} x^{-\alpha} \frac{1}{n!} \deriv^{n} \left(e^{-x} x^{n+\alpha} \right) =
\frac{x^{-\alpha}}{n!} (\deriv -1)^{n} x^{n + \alpha}.
\]
Thus,
\[
  x^{-(d-k)} f_{\rootnode}(x) = (-1)^{k} \frac{k!}{m^{k}} L_{k}^{(d-k )}(m x).
\]
We now employ a lower bound on the smallest root of associated Laguerre polynomials
  due to Krasikov \cite[Theorem 1]{KrasikovAnsatz}.

\begin{theorem}\label{thm:krasikovLaguerre}
For $\alpha > -1$,
\[
  \lambda_{k} (L_{k}^{(\alpha)}(x)) 
  \geq 
V^{2} + 3 V^{4/3}(U^{2} - V^{2})^{-1/3},
\]
where 
$V = \sqrt{k + \alpha + 1} - \sqrt{k}$ and $U = \sqrt{k + \alpha + 1} + \sqrt{k}$.
\end{theorem}
\begin{corollary}\label{cor:lagBound}
\[
 \lambda_{k}(f_{\rootnode} (x) ) > \frac{1}{m}(\sqrt{d} - \sqrt{k})^{2}.
\]
\end{corollary}
\begin{proof}
Applying Theorem \ref{thm:krasikovLaguerre}
  with $\alpha = d-k$ and thus
\[
  V = (\sqrt{d+1} - \sqrt{k})
\]
gives
\[
 \lambda_{k}(f_{\rootnode} (x) ) 
  =  \lambda_{k}(L_{k}^{(d-k)}(m x))
  = \frac{1}{m} \lambda_{k}(L_{k}^{(d-k)}(x))
  \geq V^{2}
  > \frac{1}{m}  (\sqrt{d} - \sqrt{k})^{2}.
\]
\end{proof}

Lemma \ref{lem:laguerreFormula} and Corollary \ref{cor:lagBound} are all one needs to
  establish Theorem \ref{thm:ss} in the isotropic case.
Theorems \ref{thm:cartesianInterlacing} and \ref{thm:kthRoot} 
  tell us that there exists a sequence $s_{1}, \dotsc , s_{k} \in [m]^{k}$ for which
\[
   \lambda_{k}(f_{s_{1}, \dotsc , s_{k}}) \geq 
  \lambda_{k}(f_{\rootnode}) > 
   \frac{ (\sqrt{d} - \sqrt{k})^{2}}{m} . 
\]
As $f_{s_{1}, \dotsc , s_{k}}$ is the characteristic polynomial
  of $\sum_{1 \leq i \leq k} u_{s_{i}} u_{s_{i}}^{T}$,
  this sequence must consist of distinct elements.
If not, then the matrix in the sum would have rank at most $k-1$
  and thus $\lambda_{k} = 0$.
So, we conclude that there exists a set $S \subset [m]$ of size $k$
  for which
\[
\smin(B_S)^2 =
  \lambda_{k} \left(\sum_{i \in S} u_{i} u_{i}^{T} \right)
>
   \frac{ (\sqrt{d} - \sqrt{k})^{2}}{m} 
= 
  \left(1 - \sqrt{\frac{k}{d} } \right)^{2}
  \frac{d}{m}.
\]


\section{The Nonisotropic Case and the Schatten $4-$norm}\label{sec:schatten}
In this section we prove the promised strengthening of
  Theorem \ref{thm:ss} in terms of the Schatten 4-norm.
In the proof it will be more 
  natural to work with eigenvalues of $BB^T$ rather than singular values of $B$
  (and its submatrices).
For a symmetric matrix $A$, we define
\begin{equation}\label{eqn:kappa}
  \kappa_{A} \defeq \frac{\tr(A)^2}{\tr(A^2)}.
\end{equation}
With this definition and the change of notation $A=BB^T=\sum_{i} u_i u_i^T$
  the theorem may be
  stated as follows.
\begin{theorem} \label{thm:nonisotropic} Suppose $u_1,\ldots,u_m$ are vectors with $\sum_{i=1}^m
	u_{i} u_{i}^{T} = A$. Then for every integer $k \le \kappa_A$,
there exists a set $S \subset [m]$ of size $k$ with
\begin{equation} \lambda_k\left(\sum_{i \in S} u_{i} u_{i}^{T} \right)\geq
\left(1-\sqrt{\frac{k}{\kappa_A}}\right)^2\frac{\tr(A)}{m}.\end{equation}
\end{theorem}


We prove this theorem by examining the same interlacing family as in the previous section.
As we are no longer in the isotropic case, we need to re-calculate the polynomial at the root of the tree, which
will not necessarily be a Laguerre polynomial.
We give the formula for general random vectors with finite support, but will apply it to the special
  case in which each random vector is uniformly chosen from $u_{1}, \dotsc , u_{m}$.

\begin{lemma}\label{lem:diffformula} 
Let $\X$ be a random $d$-dimensional vector with finite support.
If $\X_1,\ldots,\X_k$ are i.i.d. copies of $\X$,
  then
\[
\E  \mydet{x I - \sum_{i\le k} \X_i\X_i^T }
=
  x^{d-k}  \prod_{i=1}^{d} (1-\lambda_{i} \deriv ) x^k,
\]
where $\lambda_{1}, \dotsc , \lambda_{d}$ are the eigenvalues of 
  $\E \X \X^{T}$.
\end{lemma}
\begin{proof} 
Let $M = \E \X \X^{T}$.
By introducing variables $t_{1}, \dotsc , t_{k}$ and applying Lemma \ref{lem:randomUpdate} $k$ times,
  we obtain
\begin{align*}
\E  \mydet{x I - \sum_{i\le k} \X_i\X_i^T }
& = 
\prod_{i=1}^{k}(1 - \partial_{t_{i}})
\mydet{x I + \left(\sum_{i=1}^{k} t_{i} \right) M}  \big|_{t_{1} = \dotsb  = t_{k} = 0}
\\
& = 
\prod_{i=1}^{k}(1 - \partial_{t_{i}})
\left(
  \sum_{\ell=0}^{d} x^{d- \ell} \left(\sum t_{i} \right)^{\ell} e_{\ell}(M)
 \right)
  \big|_{t_{1} = \dotsb  = t_{k} = 0}.
\end{align*}
By computing
\[
   \prod_{i=1}^{k}(1 - \partial_{t_{i}})  \left(\sum_{i=1}^{k} t_{i} \right)^{\ell}   \big|_{t_{1} = \dotsb  = t_{k} = 0}
= 
\begin{cases}
\frac{(-1)^{\ell} k!}{(k-\ell)!}  &  \text{if $k \geq \ell$}
\\
0  & \text{otherwise},
\end{cases}
\]
we simplify the above expression to
\[
  \sum_{\ell=0}^{k} x^{d- \ell} (-1)^{\ell} \frac{k!}{(k-\ell)!} e_{\ell}(M).
\]
Since $\deriv^{\ell} x^k = x^{k-\ell} k! / (k-\ell)!$ for $\ell \leq k$
  and $\deriv^{\ell} x^{k} = 0$ for $\ell > k$,
  we can rewrite this as
\begin{align*}
x^{d-k}  \sum_{\ell=0}^{k} \deriv^{\ell} (-1)^{\ell} e_{\ell}(M) x^{k}
& =
x^{d-k}  \sum_{\ell=0}^{d} \deriv^{\ell} (-1)^{\ell} e_{\ell}(M) x^{k}
\\
& =
x^{d-k} \mydet{\deriv I - M} x^{k}
\\
& =
x^{d-k} \prod_{i=1}^d(1-\lambda_{i} \deriv )x^k,
\end{align*}
as desired.
\end{proof}

We now require a lower bound on the smallest zero of $\prod_{i=1}^{d} (1-\lambda_{i} \deriv ) x^k$.
We will use the following lemma, which tells us that the $\aminName$ of a polynomial
  grows in a controlled way
  as a function of $\lambda$ when we apply a $(1- \lambda \deriv )$ operator to it.
This is similar to Lemma 3.4 of \cite{BatsonSpielmanSrivastava}, which was written in
  the language of random rank one updates of matrices.

\begin{lemma}\label{lem:barriershift} 
If $p(x)$ is a real-rooted polynomial and $\lambda>0$, then
  $(1 - \lambda \deriv ) p(x)$ is real-rooted and
\[
\amin{(1-\lambda \deriv )p(x)}
\ge \amin{p(x)} +  \frac{1}{1/\lambda + 1/\alpha}.
\]
\end{lemma}
\begin{proof}
It is well known that $(1 - \lambda \deriv ) p(x)$ is real rooted:
  see \cite[Problem V.1.18]{PolyaSzegoII}, \cite[Corollary 18.1]{Marden}, or \cite[Lemma 3.7]{ICM14} or for a proof.
To see that $\lambda_{min}(p) \leq \lambda_{min} (p - \lambda p' ) $ for $\lambda > 0$,
  recall that $\lambda_{min}(p') \geq \lambda_{min}(p)$.
So, both $p$ and $- \lambda p'$ have the same single sign for all $x < \lambda_{min}(p)$, and thus
  $p - \lambda p' $ cannot be zero there.

Let $p$ be have degree $d$ and zeros $\mu_{d} \le \ldots\le \mu_{1}$. 
Let $b=\amin{p}$, so that
  $\Phi_{p}(b) = 1/\alpha$.
To prove the claim it is enough to show for
\[
\delta = \frac{1}{1/\lambda+1/\alpha}
\]
that $b+\delta \leq \lambda_{d}((1-\lambda \deriv )p)$ and  that
\begin{equation}\label{eqn:barrierShiftGoal}
\Phi_{(1-\lambda \deriv )p}(b+\delta) \le \Phi_{p}(b).
\end{equation}
The first statement is true because
\[
\frac{1}{\mu_{d} - b} \le \Phi_{p}(b) = 1/\alpha,
\]
so $b+\delta < b+\alpha \le \mu_{d} \leq \lambda_{d} ((1-\lambda \deriv )p)$.

We begin our proof of  the second statement
  by expressing
  $\Phi_{(1-\lambda \deriv )p}$ in terms of 
  $\Phi_{p}$ and $\Phi_{p}'$:
\begin{equation}\label{eqn:barrierupdate}
\Phi_{(1-\lambda \deriv )p} 
= -\frac{(p-\lambda p')'}{p-\lambda p'} 
= -\frac{(p(1+ \lambda\Phi_p))'}{p(1+\lambda\Phi_p)} 
= -\frac{p'}{p}-\frac{\lambda \Phi_p'}{1+\lambda\Phi_p}
=\Phi_p-\frac{\Phi_p'}{1/\lambda+\Phi_p},
\end{equation}
wherever all quantities are finite, which happens everywhere except
at the zeros of $p$ and $(1-\lambda \deriv )p$. 
Since $b+\delta$ is strictly below the zeros of both, it follows that:
\[
\Phi_{(1-\lambda \deriv ) p}(b+\delta) =
\Phi_p(b+\delta)-\frac{\Phi_p'(b+\delta)}{1/\lambda +\Phi_p(b+\delta)}.
\]
After replacing  $1/\lambda $ by $1/\delta -1/\alpha= 1/\delta-\Phi_p(b)$ and
rearranging terms (noting the positivity of $\Phi_p(b+\delta)-\Phi_p(b)$),
  we see that \eqref{eqn:barrierShiftGoal} is equivalent to
\[
\left(\Phi_{p}(b + \delta ) - \Phi_{p}(b) \right)^{2}
\leq 
\Phi_{p}' (b+\delta ) - \delta^{-1} \left(\Phi_{p}(b + \delta ) - \Phi_{p}(b) \right).
\]
We now finish the proof by expanding $\Phi_{p}$ and $\Phi_{p}'$ in terms of
  the zeros of $p$:
\begin{align*}
\left(\Phi_{p}(b + \delta ) - \Phi_{p}(b) \right)^{2}
& = 
\left(
 \sum_{i} \frac{1}{\mu_{i} -b - \delta} - \sum_{i} \frac{1}{\mu_{i} -b}
 \right)^{2}
\\
& = 
\left(
 \sum_{i} \frac{\delta }{(\mu_{i} -b - \delta)(\mu_{i} -b)}
 \right)^{2}
\\
& \leq 
\left(
 \sum_{i} \frac{\delta }{\mu_{i} -b}
 \right)
\left(
 \sum_{i} \frac{\delta }{(\mu_{i} -b - \delta)^{2} (\mu_{i} -b)}
 \right),
& \text{as all terms are positive}
\\
& \leq 
\left(
 \sum_{i} \frac{\delta }{(\mu_{i} -b - \delta)^{2} (\mu_{i} -b)}
 \right),
& \text{as $\delta \Phi_{p}(b) \leq 1$}
\\
& =
 \sum_{i} \frac{1}{(\mu_{i} -b - \delta)^{2}}
- 
 \sum_{i} \frac{1}{(\mu_{i} - b) (\mu_{i} -b - \delta)}
\\
& =
 \sum_{i} \frac{1}{(\mu_{i} -b - \delta)^{2}}
- 
 \delta^{-1} \left(
\sum_{i} \frac{1}{\mu_{i} -b - \delta}
-
\sum_{i} \frac{1}{\mu_{i} - b}
 \right) 
\\
& =
\Phi_{p}' (b+\delta ) - \delta^{-1} \left(\Phi_{p}(b + \delta ) - \Phi_{p}(b) \right).
\end{align*}
\end{proof}

\begin{theorem}\label{thm:muroot}
Let $\X$ be a random 
  $d$-dimensional  vector with finite support such that 
  $\E \X\X^T=M$, 
  let $\X_1,\ldots,\X_k$ be i.i.d. copies of $\X$,
  and let
\[
p(x) = \E  \mydet{x I - \sum_{i\le k} \X_i\X_i^T }
\]
Then
\[
\lambda_{k} (p) \ge
	\left(1-\sqrt{\frac{k}{\kappa_M}}\right)^2\tr(M),
\]
where $\kappa_M$ is defined as in \eqref{eqn:kappa}.
\end{theorem}

\begin{proof}
By multiplying $\X$ through by a constant, we may assume without loss of generality that
  $\Tr{M} = 1$.
In this case, we need to prove
\[
\lambda_{k} (p) \ge
	\left(1 - \sqrt{1- k \Tr{M^{2}}} \right)^2.
\]
Let $0 \leq \lambda_{d} \leq \dotsb \leq \lambda_{1}$ be the eigenvalues of $M$,
  so that Lemma \ref{lem:diffformula} implies
\[
  p(x) = x^{d-k} \prod_{i=1}^{d} (1 - \lambda_{i} \deriv ) x^{k}.
\]
Applying Lemma \ref{lem:barriershift} $d$ times for any $\alpha > 0$  yields
\begin{align*}
\lambda_{k} (p)
& \geq
\lambda_{k} \left(\prod_{i=1}^{d} (1 - \lambda_{i} \deriv ) x^{k} \right)
\\
& \geq
\amin{\prod_{i=1}^{d} (1 - \lambda_{i} \deriv ) x^{k} }
\\
& \geq \amin{x^k} + \sum_{i=1}^{d} \frac{1}{\lambda^{-1}_i+\alpha^{-1}}
\\&= -k\alpha+\sum_{i=1}^{d} \frac{1}{\lambda^{-1}_i+\alpha^{-1}},
\end{align*}
by Claim \ref{clm:aminMonomial}.

To lower bound this expression, observe that the function
\[
 y\mapsto \frac{1}{1+y\alpha^{-1}}
\]
is convex for all $\alpha>0$. 
Since $\sum_{i} \lambda_{i} = \Tr{M} = 1$, Jensen's inequality implies that
\[
 \sum_i \frac{1}{\lambda^{-1}_i+\alpha^{-1}}
 =\sum_i \frac{\lambda_i}{1+\lambda_i\alpha^{-1}}
  \ge \frac{1}{1+(\sum_i \lambda^2_i)\alpha^{-1}}
  =\frac{1}{1+\tr(M^2)\alpha^{-1}}.
\]
Thus, $\lambda_{k}(p(x))$ is at least
  $ -k \alpha+ 1/(1+\tr(M^2)\alpha^{-1})$, for every $\alpha>0$. 
Taking derivatives, we find that this
  expression is maximized when
\[
\alpha = 
  \Tr{M^{2}} \left(1 / \sqrt{k \Tr{M^{2}}} - 1 \right)
\]
which may be substituted to obtain a bound of 
\[
  \lambda_{k} (p) \geq  (1-\sqrt{k\tr(M^2)})^2,
\]
as desired.
\end{proof}

\subsection{A Polynomial Time Algorithm}\label{sec:alg}

We now explain how to produce the subset $S$ guaranteed by Theorem
  \ref{thm:nonisotropic} in polynomial time, up to a $1/n^c$ additive error in the
  value of $\sigma^2_k$ where $c$ can be chosen to be any constant. 
Selecting the $k$ elements of $S$ corresponds to an
interlacing family of depth $k$, whose nodes are labeled by expected
characteristic polynomials conditioned on partial assignments. Recall that the
polynomial corresponding to a partial assignment $s_1,\ldots,s_{j}\in
[m]^{j}$ is
$$ f_{s_1,\ldots,s_j}(x):=\E\chi\left(\sum_{i=1}^{j} u_{s_i}u_{s_i}^T  + \sum_{i=j+1}^k \X_i\X_i^T\right).$$
To find a full assignment with $\lambda_k(f_{s_1,\ldots,s_k})\ge
\lambda_k(f_{\rootnode})$, one has to solve $k$ subproblems of the following
type: given a partial assignment $s_1,\ldots,s_{j}$, find an index $s_{j+1}\in [m]$ such that
$\lambda_k(f_{s_1,\ldots,s_{j+1}})\ge \lambda_k(f_{s_1,\ldots,s_j})$.

We first show how to efficiently compute any partial assignment polynomial
$f_{s_1,\ldots,s_j}$. Letting $C=\sum_{i=1}^j u_{s_i}u_{s_i}^T$ and $\E \X\X^T=BB^T/m=M$ and applying
Lemma \ref{lem:randomUpdate} repeatedly, we have:
\begin{align*}
f_{s_1,\ldots,s_j}(x) &= \E \mydet{xI-C - \sum_{i=j+1}^k \X_i\X_i^T}\\
				    &=\left(\prod_{i=j+1}^k
(1-\partial_{t_i})\right)\mydet{xI-C+\sum_{i=j+1}^k t_i M}\bigg|_{t_i=0} \\
&=\left(\prod_{i=j+1}^k(1-\partial_{t_i})\right)
  \mydet{xI-C+\left(\sum_{i=j+1}^k t_i\right) M} \bigg|_{t_i=0}. \\
\end{align*}
We now observe that the latter determinant is a polynomial in
$t:=t_{j+1}+\ldots+t_k$. Since for any differentiable function of $t$ we have
$\partial_t = \partial_{t_i}$ for every $i=j+1,\ldots,k$, and the operator
$1-\partial_t$ preserves the property of being a polynomial in $t$, we may rewrite this
expression as:
$$ f_{s_1,\ldots,s_j}(x) = (1-\partial_t)^{k-j}
\mydet{xI-C+tM}\big|_{t=0}.$$
The bivariate polynomial
$$\mydet{xI-C+tM} :=\sum_{i=0}^n (-1)^ix^{n-i} e_i(C-tM)$$
has univariate polynomial coefficients $e_i(C-tM)\in \R[t]$ of degree $i\le n$,
These can be computed in polynomial time by a number of methods.

For example, we can compute the polynomials $e_i(C-tM)$ by
  exploiting the fact that the characteristic polynomial of a matrix can be computed
  in time $O(n^{\omega})$ where $\omega \leq 3$ is the matrix multiplication exponent
  \cite{keller1985fast}.
We first compute the characteristic polynomial of $C-t M$
  for $t$ equal to each of the $n$th roots of unity, in time $O(n^{\omega +1})$.
We then use fast polynomial interpolation 
  via the discrete Fourier transform 
  on the coefficients of these polynomials to recover the
  coefficients in $t$ of each $e_{i}(C-tM)$.
This takes time $O(n \log n)$ per polynomial.

Thus, in time $O(n^{\omega+1}+n^2\log n) = O(n^{\omega+1})$, we can compute the bivariate
polynomial $\mydet{xI-C+tM}$. Applying the operator
$$ (1-\partial_t)^{k-j} = \sum_{i=0}^{k-j} (-1)^{k-j-i}
\binom{k-j}{i}\partial_t^{i}$$
to each coefficient and setting $t$ to zero amounts to simply multiplying each
coefficient of each $e_i(C-tM)$ by a binomial coefficient, which can be carried
out in $O(n^2)$ time. Thus, we can compute $f_{s_1,\ldots,s_j}(x)$ in
$O(n^{\omega+1})$ time.

Given this subroutine, the algorithm is straightforward: given a partial
assignment $s_1,\ldots,s_j$, extend it to the $s_1,\ldots,s_{j+1}$ which
maximizes $\lambda_k(f_{s_1,\ldots,s_{j+1}})$. This may be done by enumerating
over all $m$ possibilities for $s_{j+1}$ and 
computing an $\epsilon-$approximation to the smallest root of
$f_{s_1,\ldots,s_{j+1}}(x)/x^{n-k}$ using the standard technique of binary search with
a Sturm sequence (see, e.g., \cite{basu}). This takes time $O(n^2\log(1/\epsilon))$ per polynomial, which is less than
the time required to compute the polynomial when $\epsilon=1/\mathrm{poly(n)}$.

The total running time to find a complete assignment is therefore
$O(kmn^{\omega+1})$. We have not made any attempt to optimize this running time,
and suspect it can be improved using more sophisticated ideas.

\section{The Isotropic Case without Replacement: Jacobi Polynomials} \label{sec:jacobi}

In this section we show how to improve Theorem \ref{thm:ss} in the isotropic case
  by constructing an interlacing family
  using subsets of vectors instead of sequences.

\begin{theorem}\label{thm:jacobi}
	Suppose $u_1,\ldots,u_m\in\R^n$ satisfy $\sum_{i\le m} u_{i} u_{i}^{T} =I$ and
	$k \le d$ is an integer. Then there exists a subset $S\subset [m]$ of size $k$ such that
\begin{equation}\label{eqn:ss}\lambda_k\left(\sum_{i\in S}  u_{i} u_{i}^{T} \right) 
\ge 
\frac{\left(\sqrt{d (m-k)} - \sqrt{k(m-d)}\right)^2}{m^2}.\end{equation}
\end{theorem}

The leaves of the tree in the interlacing family correspond to subsets of $[m]$ of size $k$.
The root corresponds to the empty set, $\emptyset$, and the other internal nodes 
  correspond to subsets of size 
  less than $k$.
The children of each internal node are its supersets of size one larger.

For each $S \subset [m]$ we define
\[
U_{S} = \sum_{i \in S} u_{i} u_{i}^T
\AND 
p_{S} (x) = \mydet{xI - U_{S}}.
\]
We label the leaf nodes with the polynomials $p_{S}(x)$.
For an internal node associated with a set $T$ of size less than $k$,
  we label that node by the polynomial
\[
f_{T}(x) = \E_{\substack{S \supset T \\ \sizeof{S} = k}} p_{S}(x),
\]
where the expectation is taken uniformly over sets $S$ of size $k$
  containing $T$.
All polynomials in the family are real and monic since they
  are averages of characteristic polynomials of Hermitian matrices.
We now derive expressions for these polynomials and prove that they
  satisfy the requirements of Definition \ref{def:if}.
We give the connection between these polynomials and Jacobi polynomials in
  Section \ref{sec:JacobiConnection}.

\begin{lemma}\label{lem:jacobiAddOne}
For every subset $T$ of $[m]$ of size $t$,
\[
  \sum_{i \not \in T} p_{T \union \setof{i}} (x) = 
  (x-1)^{- (m-d-t-1)} \deriv (x-1)^{m-d-t} p_{T} (x).
\]
\end{lemma}
The expression on the left above is a sum of polynomials and thus is clearly a polynomial.
To make it clear that the term on the right is a polynomial, 
  we observe that for all polynomials $p$ and all positive $k$,
  $\deriv (x-1)^{k} p(x)$ is divisible by $(x-1)^{k-1}$.
So, the expression on the right above is a polynomial when $m-d-t \geq 1$.
It is also a polynomial when $d+t+1 \geq  m$ because in this case
  $p_{T}(x)$ is divisible by $(x-1)^{d+t-m}$.
\begin{proof}
We begin with a calculation analogous to that in the proof of Lemma \ref{lem:laguerreFormula}:
\begin{align*}
\sum_{i \notin T} p_{T \union \setof{i}}(x) 
&= \sum_{i \notin T} \mydet{xI - U_{T} - u_i u_i^T}
\\&= \sum_{i \notin T} \mydet{xI - U_{T}}\left(1 - \Tr{ \left(xI - U_{T} \right)^{-1} u_i u_i^T} \right),
\quad  \text{by Lemma~\ref{lem:matrix_det}}
\\&= p_{T} (x) \left(m - t - \Tr{ \left(xI - U_{T} \right)^{-1} (I - U_{T})} \right),
\quad  \text{as $\sum_{i \notin T} u_{i} u_{i}^{T} = I-U_{T}$}
\\&= p_{T} (x) (m-t) -  
  p_{T} (x)\Tr{ \left(xI - U_{T} \right)^{-1} \left[ (I - xI) + (xI - U_{T}) \right]}
\\&= p_{T} (x) (m-t) -  
  d p_{T} (x) - p_{T} (x) \Tr{ \left(xI - U_{T} \right)^{-1} (I - xI) }
\\&= p_{T} (x) (m-t- d) 
   - p_{T} (x) \Tr{ \left(xI - U_{T} \right)^{-1}} (1-x)
\\&= (m-t-d) p_{T}(x) + (x-1) \deriv p_{T}(x),
\end{align*}
by Lemma \ref{lem:matrix_deriv}.
We finish the proof of the lemma by observing that
  for every function $f(x)$ and every $h \in \R$, 
\[
  \deriv (x-1)^h f(x) = h (x-1)^{h-1} f(x) + (x-1)^h \deriv f(x) .
\]
Thus,
\[
(x-1)^{-(m-d-t-1)} \deriv (x-1)^{m-d-t} p_{T}(x)
=
(m-t-d) p_{T}(x) + (x-1) \deriv p_{T}(x).
\]
\end{proof}

By applying this lemma many times, 
 we obtain an expression for the polynomials labeling internal vertices of the tree.

\begin{lemma}\label{lem:jacobiInternal}
For every $T \subset [m]$ of size $t \leq k$,
\[
f_{T}(x)
= \frac{1}{\binom{m-t}{k-t}} 
\sum_{\substack{S \supset T \\ |S| = k}} p_{S}(x)
= \frac{(m-k)!}{(m-t)!} (x-1)^{-(m-d-k )} \kderiv{k-t} (x-1)^{m-d-t} p_{T}(x).
\]
In particular,
\[
f_{\emptyset}(x)
= \frac{(m-k)!}{m!} (x-1)^{-(m-d-k)} \kderiv{k} (x-1)^{m-d} x^d.
\]
\end{lemma}
\begin{proof}
We will prove by induction on $k$ that 
\[
\sum_{\substack{S \supset T \\ |S| = k}} p_{S}(x)
= \frac{1}{(k-t)!} (x-1)^{-(m-d-k)} \kderiv{k-t} (x-1)^{m-d-t} p_{T}(x).
\]
For $k = t$, the polynomial is simply $p_{T}(x)$.
To establish the induction, observe that for $k > t$
\[
\sum_{\substack{S \supset T \\ |S| = k}} p_S(x) 
= \frac{1}{k-t} \sum_{\substack{S \supset T\\ |S| = k-1}} \sum_{t \notin S} p_{S + t}(x)
\]
and apply the inductive hypothesis and Lemma \ref{lem:jacobiAddOne} to obtain
\begin{align*}
\sum_{\substack{S \supset T \\ |S| = k}} p_S(x) 
& = \frac{1}{k-t}
(x-1)^{-(m-d- (k-1) -1)} \deriv (x-1)^{m-d- (k-1)} \times
\\
& \times
\frac{1}{(k-t-1)!} (x-1)^{-(m-d-(k-1))} \kderiv{k-1-t} (x-1)^{m-d-t} p_{T}(x)
\\
& = \frac{1}{(k-t)!}
(x-1)^{-(m-d-k)} 
 \kderiv{k-t} (x-1)^{m-t} p_{T}(x).
\end{align*}
\end{proof}

\begin{theorem}\label{thm:subsetInterlace}
The polynomials $p_{S}(x)$ for $\sizeof{S} = k$ are an interlacing family.
\end{theorem}
\begin{proof}
As explained above, the internal nodes of the tree are the polynomials
  $f_{T}(x)$ for $T \subset [m]$ and $\sizeof{T} < k$.
By definition these polynomials satisfy condition $a$ of an interlacing family.
We now show that they also satisfy condition $b$.

Let $T \subset [m]$ have size less than $k$.
We must prove that for every $i$ and $j$ not in $T$ and all $0 \leq \mu \leq 1$,
\[
  f_{\mu}(x) \defeq   \mu  f_{T \union \setof{i}}(x) + (1-\mu ) f_{T \union \setof{j}}(x)
\]
is real-rooted.

As Claim \ref{clm:cauchyInterlacing} tells us that 
  $p_{T \union \setof{i}}$ and $p_{T \union \setof{j}}$
  have a common interlacing,
  Theorem \ref{thm:Fisk}
  implies that 
\[
p_{\mu}(x) \defeq \mu p_{T \union \setof{i}}(x) + (1-\mu ) p_{T \union \setof{j}}(x)
\]
  is real rooted.
Lemma \ref{lem:jacobiInternal} implies that
\[
f_{\mu}(x)
=
\frac{(m-k)!}{(m-t)!} (x-1)^{-(m-d-k )} \kderiv{k-t} (x-1)^{m-d-t} p_{\mu} (x).
\]
So, we can see that $f_{\mu}(x)$ is real rooted by observing that 
  real rootedness is preserved by 
  multiplication by $(x-1)$, taking derivatives, and dividing by
  $(x-1)$ when $1$ is a root.
\end{proof}

We begin proving a lower bound on the $k$th largest root of
  $f_{\emptyset}(x)$ by expressing it as the smallest root of a simpler polynomial.

\begin{lemma}\label{lem:jacRootTransform}
The $k$th largest root of $f_{\emptyset} (x)$ is equal to the smallest root of the polynomial
\begin{equation}\label{eqn:lem:jacRootTransform}
\kderiv{d} (x-1)^{m-k}x^k.
\end{equation}
\end{lemma}
\begin{proof}
As all the eigenvalues of $U_{S}$ for every subset $S$ are less than $1$
  and because $U_{S}$ is positive semidefinite, all the zeros of $p_{S}(x)$ are between
  $0$ and $1$.
As all polynomials $p_{S}(x)$ are monic, they are all positive for $x > 1$ and thus
  $f_{\emptyset}(x)$ is as well.
This argument, and a symmetric one for $x < 0$, implies that all the zeros of $f_{\emptyset}(x)$
  are between $0$ and $1$.

The polynomial $f_{\emptyset}(x)$ has at least $d-k$ zeros at 0.
So, its $k$th largest root is the smallest root of
\[
 x^{-(d-k)} (x-1)^{-(m-d-k)} \kderiv{k} (x-1)^{m-d} x^d.  
\]
As all the zeros of this polynomial are between $0$ and $1$, its smallest root
  is also the smallest root of
\[
 x^{-(d-k)} \kderiv{k} (x-1)^{m-d} x^d.  
\]
We now show that this latter polynomial is a constant multiple of
  \eqref{eqn:lem:jacRootTransform}.

\begin{align*}
\kderiv{k} (x-1)^{m-d} x^{d}
& = 
 \sum_{i=0}^{k} \binom{k}{i} 
     \left(\kderiv{k-i} (x-1)^{m-d} \right)
     \left(\kderiv{i} x^{d} \right)
\\
& = 
 \sum_{i=0}^{d} \binom{k}{i} 
     \left(\kderiv{k-i} (x-1)^{m-d} \right)
     \left(\kderiv{i} x^{d} \right),
\text{as $\kderiv{i} x^{d} = 0$ for $i > d$,}
\\
& = 
 \sum_{i=0}^{d} \frac{k!}{i! (k-i)!}
  \frac{(m-d)! (x-1)^{m-d-k+i}}{(m-d-k+i)!}     
  \frac{d! x^{d-i}}{(d-i)!}
\\
& = 
 \frac{(m-d)!}{(m-k)!}
x^{d-k}
 \sum_{i=0}^{d} \frac{d!}{i! (d-i)!}
  \frac{(m-k)! (x-1)^{m-d-k+i}}{(m-d-k+i)!}     
  \frac{k! x^{k-i}}{(k-i)!}
\\
& = 
 \frac{(m-d)!}{(m-k)!}
x^{d-k}
\kderiv{d} (x-1)^{m-k}x^{k}.
\end{align*}
\end{proof}

We use the following lemma to prove a lower bound on the smallest
  zero of the polynomial in \eqref{eqn:lem:jacRootTransform}.
This lemma may be found in \cite[Lemma 4.2]{MSSfinite}, or may be proved
  by applying Lemma  \ref{lem:barriershift} in the limit as
  $\lambda $ grows large.

\begin{lemma}\label{lem:amin-deriv}
For every real rooted polynomial $p (x)$ of degree at least two and $\alpha > 0$,
\[
	\amin{\deriv p(x)} \geq \amin{p (x)} + \alpha .
\]
\end{lemma}

\begin{theorem}
For $m > d > k$,
\label{thm:norepmin}
\begin{equation}
\label{jbound}
\lambda_k(f_{\emptyset}(x)) 
> \frac{\left(\sqrt{d(m-k)} - \sqrt{k(m-d)}\right)^2}{m^2} .
\end{equation}
\end{theorem}
\begin{proof}
One can check by substitution that
\[
\amin{(x-1)^{m-k} x^k} 
= \frac{(1- \alpha m ) - \sqrt{(1 - \alpha m)^2 + 4 \alpha k}}{2} .
\]
Applying Lemma \ref{lem:amin-deriv} $d$ times gives
\[
\amin{ \deriv^{d} (x-1)^{m-k} x^k} 
\geq \frac{(1- \alpha m) - \sqrt{(1 - \alpha m)^2 + 4 \alpha k}}{2} + d \alpha .
\]
Define
\begin{equation}\label{eqn:norepmin}
  u(\alpha ) = \frac{(1- \alpha m) - \sqrt{(1 - \alpha m)^2 + 4 \alpha k}}{2} + d \alpha.
\end{equation}
We now derive the value of $\alpha$ at which $u(\alpha )$ is maximized.

Taking derivatives in $\alpha$ gives
\[
u'(\alpha) 
= d - \frac{m}{2} - \frac{2k + \alpha m^2 - m}{2\sqrt{(1 - \alpha m)^2 + 4 \alpha k}}.
\]
Notice that 
\[
u'(0)
= d - k \geq 0
\AND
\lim_{\alpha \to \infty} u'(\alpha) 
= d - m \leq 0.
\]
By continuity, a maximum will occur at a point $\alpha^* \geq 0$ at which $u'(\alpha^*) = 0$.
The solution is given by
\[
\alpha^* 
= \frac{m-2k}{m^2} - \frac{m-2d}{m^2}\sqrt{\frac{k(m-k)}{d(m-d)}},
\]
which is positive for $m > d > k$.

After observing that
\[
(1-\alpha^{*}m)^{2} + 4 \alpha^{*} k = \frac{k (m-k)}{d (m-d)},
\]
we may plug $\alpha^{*}$ into the definition of $u$ to obtain
\begin{align*}
u(\alpha^*) 
& = \frac{d}{m} + \frac{k}{m} - \frac{2dk}{m^2}
   + \frac{1}{m^2}\left(2 d^{2} - 2 d m\right)\sqrt{\frac{k(m-k)}{d(m-d)}}
\\
& =  \frac{1}{m^2} \left(dm + km - 2dk
  - 2 \sqrt{k(m-k) d (m-d)}
  \right)
\\
& =  \frac{1}{m^2} \left(
  \sqrt{d(m-k)} - \sqrt{k(m-d)}
  \right)^{2}.
\end{align*}

\end{proof}

\subsection{Jacobi polynomials}\label{sec:JacobiConnection}

Rodrigues' formula for the degree $n$ Jacobi polynomial with parameters $\alpha$ and $\beta$ is
\cite[(4.3.1)]{Szego}
\[
P_{n}^{(\alpha, \beta)}(x) =
\frac{(-1)^{n}}{2^{n} n!}
 (1-x)^{-\alpha}(1+x)^{-\beta} \kderiv{n} (1-x)^{\alpha + n}(1 + x)^{\beta + n}.
\]
These are related to the polynomials $f_{\emptyset}(x)$ from the previous section
  by
\[
  P_{k}^{(m-d-k,d-k)}(2 x - 1) = \binom{m}{k} x^{-(d-k)} f_{\emptyset}(x).
\]

Thus, lower bounds on the $k$th smallest root of $f_{\emptyset} (x)$ translate into lower
  bounds on the smallest root of Jacobi polynomials.
By using the following claim, we can improve the lower bound from Theorem \ref{thm:norepmin}
  to
\begin{equation}\label{eqn:lowerJacImproved}
  \frac{\left(\sqrt{(d+1) (m-k)} - \sqrt{k(m-d-1)}\right)^2}{m^2}.
\end{equation}

\begin{claim}\label{clm:amin}
For every real-rooted polynomial $p(x)$ and $\alpha > 0$,
\[
  \amin{p} + \alpha \leq   \lmin (p).
\]
\end{claim}
\begin{proof}
We first observe that there is a $z < \lmin (p)$ for which
  $\Phi_{p}(z) = 1/\alpha$.
This holds because $\Phi_{p}(z)$ is continuous for $z < \lmin (p)$,
  approaches infinity as $z$ approaches $\lmin (p)$ from below,
  and approaches zero as $z$ becomes very negative.
For a $z < \lmin $ such that $\Phi_{p}(z) = 1/\alpha$, we have
\[
  1/\alpha = \Phi_{p}(z) \geq \frac{1}{ \lmin - z}.
\]
The claim follows.
\end{proof}

The bound that \eqref{eqn:lowerJacImproved} implies for Jacobi polynomials
  seems to be incomparable to the bound obtained by Krasikov \cite{KrasikovAnsatz},
  although numerical evaluation suggests that Krasikov's bound is usually stronger.

\bibliographystyle{alpha}
\bibliography{IFbib}

\end{document}